\documentclass[a4paper,11pt,reqno,noindent]{amsart}
\usepackage[centertags]{amsmath} 
\usepackage{amsfonts,amssymb,amsthm}
\usepackage{hyperref}
\usepackage{mathtools}

 \hypersetup{pdfmenubar=false, 
pdfnewwindow=true, 
colorlinks=false, 
linkcolor=blue, 
citecolor=blue, 
filecolor=magenta, 
urlcolor=cyan 
} 
\usepackage{graphicx}

\usepackage[english]{babel} \usepackage{newlfont} \usepackage{color}
\usepackage[body={15cm,21.5cm},centering]{geometry} \usepackage{fancyhdr}
 \usepackage{esint} \usepackage{enumerate}

\usepackage[active]{srcltx}
\usepackage{subcaption}

\newtheorem{theorem}{Theorem} \newtheorem{lemma}{Lemma}
\newtheorem{proposition}{Proposition} 
\newtheorem{corollary}{Corollary} 
\newtheorem{definition}{Definition}

\theoremstyle{definition} 
 
\newtheorem{remark}{Remark}

\newcommand{\dist}{\operatorname{dist}} 

\newcommand{\e}{\varepsilon}

\newcommand{\R}{\mathbb{R}}

\renewcommand{\d}{\mathrm{d}} \renewcommand{\L}{\mathbb{L}}
\newcommand{\A}{\mathcal{A}} 
 \newcommand{\id}{\mathrm{Id}}

 \renewcommand{\H}{\mathcal{H}}
 
\renewcommand{\L}{{\mathcal L}}

\newcommand{\cof}{\mathrm{cof}\,}

\title{Variational competition between  the full Hessian and its determinant  for convex functions}
\date{\today} 

\author[P. Gladbach]{Peter Gladbach}
\author[H. Olbermann]{Heiner Olbermann} 
\address[Peter Gladbach]{Universit\"at Bonn, 53115 Bonn, Germany}
\address[Heiner Olbermann]{UCLouvain, 1348 Louvain-la-Neuve, Belgium}
\email[Peter Gladbach]{gladbach@iam.uni-bonn.de}
\email[Heiner Olbermann]{heiner.olbermann@uclouvain.be}

\begin{document}
\maketitle

\begin{abstract}

  We prove upper and lower bounds for a variational functional for convex functions satisfying certain boundary conditions on a sector of the unit ball in two dimensions. The functional contains two terms: The full Hessian and its determinant, where the former is treated as a small perturbation in the space $L^2$ and the latter as the leading-order term, in the negative Sobolev space $W^{-2,2}$. We point out how this setting is motivated by problems in nonlinear elasticity, and obtain a corollary for a variational problem based on the so-called F\"oppl-von K\'arm\'an energy.

  \end{abstract}

\section{Introduction}

\subsection{Statement of main result}
Let $S\subset B(0,1)\subset \R^2$ be a sector of opening angle  $2\alpha<\pi$,
\[
  S=\{r(\cos\varphi,\sin\varphi):0\leq r<1,-\alpha<\varphi<\alpha\}\,.
  \] We denote by $\Gamma$ the curved part of the boundary of $S$, $\Gamma=\{x\in \partial S:|x|=1\}$.
Consider $v\in W^{2,2}(S)$ with the partial boundary conditions
  \begin{equation}\label{eq:6}
  \left.\begin{split}
  v(x)&=1\\ \nabla v(x)&=x
\end{split}\right\}\text{ for }x\in\Gamma\,.
\end{equation}
Note that \eqref{eq:6} is satisfied by  $v_0(x):=|x|$. We set $2S=\{2x:x\in S\}$, and have that every function $v\in W^{2,2}(S)$ that satisfies \eqref{eq:6} may be extended to a function $Ev\in W^{2,2}(2S)$ by setting
\[
  Ev(x):=\begin{cases} v(x) &\text{ if } x\in S\\ |x| &\text{ else}.\end{cases}
    \]
Now we  define the negative Sobolev norm $\|f\|_{W^{-2,2}(2S)}$ by
\[
  \|f\|_{W^{-2,2}(2S)}=\sup\left\{ \int_{2S} \varphi f \d x: \varphi\in W^{2,2}_0(2S), \|\varphi\|_{W^{2,2}_0(2S)}\leq 1\right\}\,,
  \]
  where (as usual) $W^{2,2}_0(2S)$ denotes the closure of the set of functions $\varphi\in C^\infty_c(2S)$ with respect to the norm $\|\varphi\|_{W^{2,2}_0(2S)}=\left(\int_{2S}|\nabla^2\varphi|^2\d x\right)^{1/2}$. The space $W^{-2,2}(2S)$ may then be defined as  the completion of $C^{\infty}(2S)$ (say) with respect to the norm $\|\cdot\|_{W^{-2,2}(2S)}$.
  For $h>0$, we define the  variational functional
  \[
    J_h(v)= h^{-2}\|\det \nabla^2(Ev)\|_{W^{-2,2}(2S)}^2+\|\nabla^2 v\|^2_{L^2(S)}\,.
    \]
    We restrict the set of permissible functions to those that are \emph{convex},
    \[
      \mathcal A=\left\{v\in W^{2,2}(S), v\text{ is convex and satisfies }\eqref{eq:6}\right\}\,.
\]
Our main theorem is as follows:

  \begin{theorem}\label{thm:mainma}
There exist  numerical constants $h_0,C>0$ such that for every $h\in(0,h_0)$,
    \[
      2\alpha \log \frac{1}{h}-C\log\log\frac{1}{h}\leq \inf_{v\in \mathcal A} J_h(v)\leq 2\alpha \log\frac{1}{h}+C\,.
      \]
  \end{theorem}

  \begin{remark}
    \begin{itemize}
    \item[(i)] Let us restate the lower bound in Theorem \ref{thm:mainma} in a slightly different fashion, emphasizing its relation to  the Monge-Amp\`ere equation in two dimensions: Suppose that $v\in \A$ satisfies $\det\nabla^2 v=\mu$ on $S$, where $\mu$ is a Radon measure. Then
      \[
        h^{-2}\|\mu\|_{W^{-2,2}}^2+\|\nabla^2 v\|^2_{L^2}\geq 2\alpha \log h^{-1}-\text{ lower order terms}\,.
      \]
       We point out that  our estimate  does not require that the right hand side of the equation $\mu$ is bounded away from zero, as
is the case, e.g., for the well known estimates in regularity theory for the Monge-Amp\`ere equation \cite{caffarelli1990interior,de2013w}.
    \item[(ii)] Minimization of the  energy functional $J_h$ may be thought of as a variational relaxation of  the constraint $\det \nabla^2 v=0$. The choice to measure the discrepancy of the Hessian determinant from 0 in $W^{-2,2}$ is motivated by nonlinear elasticity, see Section \ref{sec:nonlinear-elasticity} below.
    \item[(iii)] The reason why we consider  this particular setup is the following: 
      The only solution $v\in W^{2,2}_{\mathrm{loc}}(S)$ of $\det\nabla^2v=0$ in $S$ with the boundary conditions \eqref{eq:6} is given by $v_0(x)=|x|$, which in turn is not an element of $W^{2,2}(S)$. One expects minimizers of $J_h$ to be close to $v_0$ except at points close to the origin (distance at most $h$), where some smoothing should occur. Indeed, 
an explicit
construction of a competitor based on this idea straightforwardly leads to
the upper bound  $2\alpha \log \frac{1}{h}+C$ (see the proof of Corollary \ref{cor:FvK}).
The function $v_0$ is not a good competitor because of the non-integrability of $|D^2v_0|^2$ near the origin. For the lower bound, it has to be understood how  removing the non-integrability affects the membrane energy  $\|\det\nabla^2(Eu)\|_{W^{-2,2}}$. The idea of how to quantify this interplay is sketched in  a non-technical manner in Section \ref{sec:idea}. We note that our arguments would apply to any other choice of domain and boundary conditions that  lead to the same type of non-integrability of the square of the second gradient for ``flat'' configurations (i.e., functions $v$ satisfying $\det \nabla^2 v=0$). We stick to the convex sector $S$ and the clamped boundary conditions \eqref{eq:6} for simplicity and definiteness. In particular, we consider only convex domains, in order to work with a self-evident notion of convex function.

    \end{itemize}
  \end{remark}


\subsection{Scientific context; motivation for the definition of $J_h$}
\subsubsection{Rigidity and $h$-principle} By a  classical 
 result,  any  homogeneous solution of the two-dimensional Monge-Amp\`ere
equation is \emph{rigid}. By this we mean the following  well-known dichotomy for solutions $v\in C^2(\Omega)$ of 
$\det \nabla^2 v=0$, where $\Omega\subset\R^2$: For  any $x\in \Omega$, either $\nabla v$ is constant in a
neighborhood of $x$, or there exists a line segment through $x$ ending on the
boundary of $\Omega$  on which $\nabla v$ is constant.\footnote{It has been 
shown by Korobkov \cite{korobkov2009properties} that the rigidity of $v\in C^2$ with $\det\nabla^2 v=0$ still holds for $C^1$ functions $v$ whose gradient
image $\nabla v(\Omega)$ is
essentially one-dimensional  (see also \cite{korobkov2007properties}). Imposing the boundary conditions \eqref{eq:6},  one expects the image of the gradient to
have small distance from a  one-dimensional curve. The present article
is partly inspired by Korobkov's work and his analysis  of the preimages $(\nabla
v)^{-1}(y)$ where $\nabla v(\Omega)$ is   one-dimensional; we will consider instead preimages $(\nabla v)^{-1}(A)$ of appropriately
chosen sets $A$ that are contained in the 
image $\nabla v(\Omega)$ which has small distance from a one-dimensional curve.}

\medskip

We will now stray into differential geometry,  considering embedded surfaces  $y(\Omega)\subset \R^3$  instead of graphs $\{(x,v(x)): x\in\Omega\}$, even though it is the latter that we are actually interested in. We do so because the results that we cite as motivation for our setup are  better known in the geometric setting. 

\medskip

In this differential geometric setting, the analogue of the rigidity from above   is the following. Suppose that  $M=y(\Omega)\subset\R^3$ is a  two-dimensional  $C^2$ embedded surface with vanishing Gauss curvature $K$. Then for  every point $x\in M$ either there exists a neighborhood in which $M$ is flat (vanishing second fundamental form) or there exists a straight line segment $L$ contained in $M$ whose endpoints are contained in the boundary of $M$.
As an aside, we note the well known fact that zero Gauss curvature is equivalent to the isometry of the embedding,  i.e.,
\[
  K=0 \text{ on }M \quad \Leftrightarrow \quad \nabla y^T(x)\nabla y(x)=\id_{2\times 2} \quad \forall \,x\in \Omega\,,
  \]
  and we see that this condition can be formulated without any problem also for embeddings of regularity $C^1$.

\medskip

When we do lower the regularity requirements and accept   $C^1$ isometric embeddings, the situation changes dramatically and all rigidity is lost. By a theorem by Nash \cite{MR0065993} and Kuiper \cite{MR0075640} there exists a $C^1$ isometric embedding  in every $C^0$-neighborhood of any \emph{short} embedding, where an embedding $y:\Omega\to \R^3$ is called short if $\nabla y^T\nabla y< \id_{2\times 2}$ in the sense of positive definite matrices. The Nash-Kuiper Theorem  is an instance of the so-called $h$-principle \cite{MR864505}, which -- using  somewhat imprecise language -- can   be thought of as the opposite  of rigidity.

\medskip

Obviously, the isometry condition  is  local, but rigidity is a statement about the global shape of the surface. These considerations  give rise to the following question: How large do second gradients of $y$ have to be  if
\begin{itemize}
\item[a)] $y$ is almost an isometry
\item[b)] $y$ is prevented from actually being an isometry  by the imposition of constraints on the  shape of its graph but such that
 \item[c)] these constraints do allow for  short embeddings?
  \end{itemize}
     The natural way to make this question rigourous is in a variational framework. For example, one may consider the variational functional
\begin{equation}
I_h(y)=\int_{\Omega} \left(|\nabla y^T\nabla y-\id_{2\times 2}|^2+h^2|\nabla^2y|^2\right)\d x\,,\label{eq:14}
\end{equation}
where $h>0$ is interpreted as a small parameter, and introduce constraints (such as boundary conditions or obstacles) on $y$ that prohibit smooth isometries while allowing for short embeddings. Then the -- very difficult -- task is to find  lower bounds for $I_h$ that are  close to optimal (optimal in in the leading order of $h$, say).
Without the introduction of further strong assumptions, it is not known how to treat such questions. Even if such assumptions are added, it is in general  still  a non-trivial task to carry out a quantitative analysis.

\subsubsection{Nonlinear elasticity} \label{sec:nonlinear-elasticity} Over the last decades there has been a lot of interest in the shape formation of thin elastic sheets subject to constraints and/or under the influence of external forces \cite{CCMM,MR2023444,1997PhRvE..55.1577L,RevModPhys.79.643,Cerda08032005}.   In the mathematical literature, such questions have predominantly been discussed in a variational framework by investigating  boundary value  and obstacle problems for functionals as in \eqref{eq:14} or similar ones, see e.g.~\cite{MR3179665,MR1921161,2015arXiv151207416B,kohn2013analysis,MR2358334}. The papers \cite{MR3102597,conti2017symmetry} treat scaling laws of approximately conical configurations.

These works mainly focus on  different constraints than the ones characterized by b) and c) above (with the notable exception of \cite{MR2358334}). 
Some progress on this   case has been obtained recently in
\cite{2015arXiv150907378O,olber2017coneconv,olber2019crump}, where  different simplifying assumptions or modifications of the integral functional have been made. These works are
based on the
observation  that the ``membrane term'' in the functional \eqref{eq:14},
  $\|\nabla y^T\nabla y-\id_{2\times 2}\|_{L^2}^2$,
can be thought of as a penalization of (suitably linearized) Gauss curvature in a negative Sobolev space.  The information
contained in the boundary values (or, in the case of \cite{2015arXiv150907378O,olber2017coneconv}, metric defects) has to be combined with the
smallness of the Gauss curvature in  the appropriate negative Sobolev space to
obtain a lower bound for the bending energy $h^2\|\nabla^2y\|_{L^2}^2$.
This approach is different to the one used e.g.~in the works \cite{MR3168627,MR3102597,MR2358334}, where ``tensile'' Dirichlet boundary conditions are being considered. In such a setting, any deviation from the configuration that satisfies the boundary conditions with vanishing membrane bending energy can be shown to cause  energetically costly stretching of the sheet. 

The present work can be viewed as a continuation of the program started in \cite{2015arXiv150907378O,olber2017coneconv,olber2019crump}. In particular, the relation between membrane term and negative Sobolev norms motivates  our choice of the leading order term $\|\det\nabla^2 v\|_{W^{-2,2}}^2$ in the definition of $J_h$ above. For a more explicit statement of the relation between these two objects, see the proof of Corollary \ref{cor:FvK} below.


\medskip

\subsection{A corollary for the F\"oppl-von K\'arm\'an energy}
Instead of the functional
\eqref{eq:14}, a popular choice of variational functional for modeling thin elastic sheets is the so called  F\"oppl-von K\'arm\'an energy, for the definition of which one decomposes the deformation into
an in-plane component $u:\Omega\to \R^2$, and an out-of-plane component
$v:\Omega\to\R$,  $y(x)=x+(u(x),v(x))$.
For $h>0$, $u\in W^{1,2}(S;\R^2)$, and  $v\in W^{2,2}(S)$,  the F\"oppl-von K\'arm\'an energy is given by
  \begin{equation}\label{eq:3}
E_h(u,v)=\int_S \frac14|\nabla u+\nabla u^T+\nabla v\otimes\nabla v|^2+h^2|\nabla^2 v|^2\d x\,.
\end{equation}
One observes that
the components are treated
non-symmetrically in the energy. The F\"oppl-von K\'arm\'an energy can be derived heuristically for small deformations from the geometrically fully nonlinear energy \eqref{eq:14} and can be deduced rigorously from three-dimensional finite elasticity \cite{MR2210909}.

\medskip

The dichotomy between rigidity and $h$-principle  can be established also in this approximation, see \cite{lewicka2017convex}. An ``isometry'' has now to be understood as a pair $(u,v)$ that satisfies $\nabla u+\nabla u^T+\nabla v\otimes\nabla v=0$, and a ``short map'' with respect to some given symmetric matrix-valued function $A:\Omega\to \{M\in\R^{2\times 2}:M^T=M\}$ has to be understood as a pair $(u,v)$ satisfying $\nabla u+\nabla u^T+\nabla v\otimes\nabla v\leq A$ in the sense of positive definite matrices.  Again it is not known in the general case how to  establish optimal lower bounds for constraints that satisfy the conditions b) and c) above.

\medskip

In the present paper, we show how under the assumption that $v$ is convex, such an optimal bound may be proved. Of course the Nash-Kuiper-like oscillatory deformations are a priori eliminated by this assumption, and the reader might wonder what the interest is in such a result once this major difficulty is removed. The answer to that question is that even under the assumption of convexity it is still a non-trivial task to establish \emph{quantitative} results on the variational competition between membrane and bending terms (or, in the case of $J_h$, between the Hessian determinant and the full Hessian).  

\medskip

For pairs $(u,v)$ that represent in-plane and out-of-plane deformations, the appropriate boundary conditions are
         \begin{equation}
       \label{eq:7}\left.\begin{split}
       v(x)&=1 \\\nabla v(x)&= x\\
u(x)&=       \frac12 \left( \arg (x) x^\bot-x\right)\,\,
\end{split}\right\}\text{ for } x\in \Gamma\,,
\end{equation}
where $x^\bot=(-x_2,x_1)$, and $\arg: B(0,1)\setminus \{(x_1,0):x_1\leq 0\}\to(-\pi,\pi)$ is the  function that maps $r(\cos\varphi,\sin\varphi)$ to $\varphi$ for $r\in (0,1)$, $\varphi\in (-\pi,\pi)$. Note that the boundary conditions \eqref{eq:7} are precisely those of the  deformation $(u_0,v_0)$ given by $u_0=\frac12\left(\arg(x) x^\bot-x\right)$, $v_0(x)= |x|$. Also note that  $\nabla u_0+\nabla u_0^T+\nabla v_0\otimes\nabla v_0=0$, but $(u_0,v_0)$ is not permissible since $v_0\not\in W^{2,2}(S)$ (as already observed in the remarks after Theorem \ref{thm:mainma}). See Figures \ref{fig:domain}, \ref{fig:cone}.

\begin{figure}[h]
  \begin{center}
    \includegraphics[height=5cm]{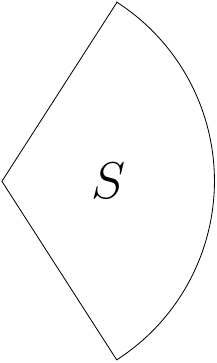}
  \end{center}
  \caption{The elastic sheets in its reference configuration.\label{fig:domain}}
  \end{figure}
\begin{figure}[h]
\includegraphics[height=5cm]{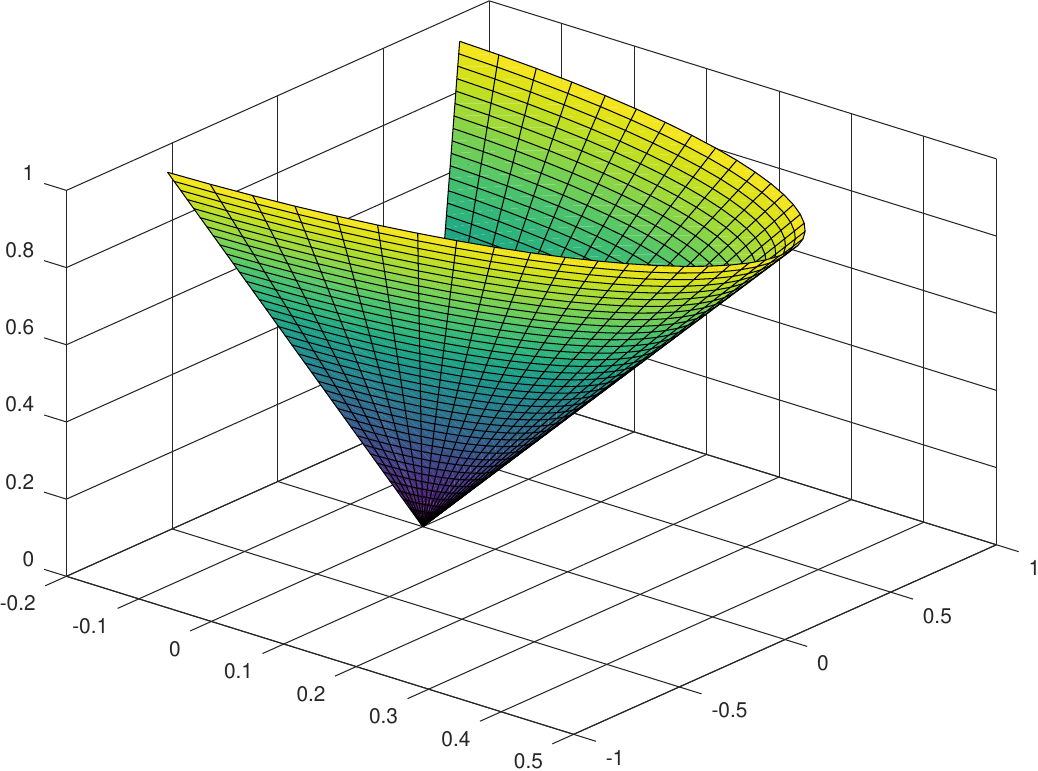}
\caption{The conically deformed sheet, satisfying the right boundary conditions
 with zero membrane and infinite bending energy. \label{fig:cone}} 
\end{figure}

The  set of permissible deformations is now given by
\newcommand{\ANE}{\mathcal A_{\mathrm{FvK}}}

       \[
         \ANE:=\left\{(u,v)\in W^{1,2}(S;\R^2)\times W^{2,2}(S): v \text{ convex }, (u,v) \text{ satisfy }\eqref{eq:7}\right\}\,.
         \]


\bigskip       

As a corollary of Theorem \ref{thm:mainma}, we have:
 \begin{corollary}
   \label{cor:FvK}
         There exist  numerical constants $C,h_0>0$ such that for $h\in(0,h_0)$, we have that
         \[
           2\alpha h^2\log\frac{1}{h}-Ch^2\log\log\frac{1}{h}\leq
           \min_{(u,v)\in\ANE} E_h(u,v)\leq 2\alpha h^2\left(\log\frac{1}{h}+C\right).
           \]
       \end{corollary}

    
        
\newcommand{\curl}{\mathrm{curl}\,}
       
\begin{proof} First we show the upper bound, which does not follow from Theorem \ref{thm:mainma}, but whose proof is rather straightforward.
           Set $u(x) = \frac12(\arg (x) x^\bot - x)$ and, for every $h>0$, set
         \[
         v_h(x) = \begin{cases}
                    |x| &\text{, if }|x|\geq h\\
                    \frac{|x|^2}{2h} + \frac{h}{2}&\text{, if }|x|<h.
                  \end{cases}
         \]
         We see that $\nabla u(x) + \nabla u^T(x) = -\frac{x\otimes x}{|x|^2}$,
         \[
         \nabla v_h(x) = \begin{cases}
                         \frac{x}{|x|}&\text{, if }|x|\geq h\\
                         \frac{x}{h}&\text{, if }|x|<h
                       \end{cases},
         \]
         and 
         \[
         \nabla^2 v_h(x) = \begin{cases}
                         \frac{x^\bot \otimes x^\bot}{|x|^3}&\text{, if }|x|\geq h\\
                         \frac{\id}{h}&\text{, if }|x|<h
                       \end{cases}.
         \]
         
         In particular, the boundary values are satisfied. We find
         \begin{equation}
         \int_S |\nabla u + \nabla u^T + \nabla v_h \otimes \nabla v_h|^2 \d x \leq \int_0^h \int_{-\alpha}^\alpha r \d\alpha\d r = 2\alpha h^2.
         \end{equation}
         
         On the other hand, the bending energy can be calculated precisely:
         \begin{equation}\label{eq: FvK hessian}
         h^2 \int_S |\nabla^2 v_h|^2\d x = 2\alpha h^2 \left( \log \frac{1}{h} + 2\right).
         \end{equation}
         
         This shows the upper bound.

         \medskip

         For the lower bound, 
      we observe that we have the distributional identity
  \begin{equation}\label{eq: curl curl}
    \det\nabla^2 v=-\frac12\curl\curl \left(\nabla u^T+\nabla u+\nabla v\otimes
    \nabla v\right)\,,
  \end{equation}
  where $\curl (w_1,w_2)=\partial_1 w_2-\partial_2 w_1$, and one of the curls is taken  column-wise, the other one row-wise.
  Thus we may 
  integrate by parts twice in the definition of $\|\det\nabla^2 Ev\|_{W^{-2,2}}$, and obtain
  
  \begin{equation}
    \begin{split}\label{eq:10}
      \|\det \nabla^2 (Ev)\|_{W^{-2,2}(2S)}&= \sup\left\{\int_{2S}\tilde\Phi\det \nabla^2
        (Ev)\d x:\,\tilde \Phi\in C_c^\infty(2S),\|\nabla^2
        \tilde \Phi\|_{L^2}\leq 1\right\}\\
      &=\sup\Big\{\int_{2S}\frac12\mathrm{Tr}\left[\left(\nabla u^T+\nabla u+\nabla (Ev)\otimes
          \nabla (Ev)\right)\cof \nabla^2\tilde \Phi\right]\d x:\\
        &\quad\tilde \Phi\in
      C_c^\infty(2S),\|\nabla^2
      \tilde \Phi\|_{L^2}\leq 1\Big\}\\
    &\leq\frac12 \|\nabla u+\nabla u^T+\nabla v\otimes\nabla v\|_{L^2(S)}\,,
  \end{split}
\end{equation}

where ``$\cof$'' denotes the cofactor matrix.
  Hence we obtain using Theorem \ref{thm:mainma} 
  \[
    \begin{split}
      E_{h}(u,v)&\geq  \|\det\nabla^2 Ev\|_{W^{-2,2}}^2+h^2\|\nabla^2 v\|_{L^2}^2\\
      &=h^2 J_{h}(v)\\
    &\geq  
 h^2 \left(2\alpha\log \frac1h-C\left(\log\log\frac1h\right)\right)\,,
\end{split}
\]
where the last inequality holds for $h_0,C$ chosen small and large enough respectively. This proves our claim.
\end{proof}

\begin{remark}
  It was in order to be able to carry out the calculation for the lower bound \eqref{eq:10} that we defined the energy in terms of the extension $Ev$, and on the enlarged domain $2S$. It is obvious that an analogous upper bound for the membrane energy in terms of $\|\det\nabla^2(Ev)\|_{W^{-2,2}}$ cannot exist, since bad choices of the in-plane-deformation $u$ can make it arbitrarily large, while they do clearly not modify $\|\det \nabla^2 (Ev)\|_{W^{-2,2}}$.
\end{remark}

\subsection{Sketch of proof of Theorem \ref{thm:mainma}}
\label{sec:idea}
The idea for the proof of the lower bound in  Theorem \ref{thm:mainma} is the following:
We identify $v$ with its extension to $2S$, where $v(x)=|x|$ for $x\in 2S\setminus S$.
We need a judicious choice of $\tilde \Phi$ to bound from below the negative Sobolev norm by
  \begin{equation}\label{eq:5}
  \int_{2S}\tilde\Phi\det \nabla^2
  v\d x\,,
\end{equation}
where  $\tilde \Phi\in C_c^\infty(2S)$ with $\|\nabla^2
        \tilde \Phi\|_{L^2}\leq 1$. We will see that up to controlled factors
        and away from the boundary,
        such a choice is given by $\tilde\Phi(x)=|x|$. Let $\tilde S\subset 2S$ be defined by $\tilde S=\{r(\cos\varphi,\sin\varphi):\varphi\in I\subset(-\alpha,\alpha),r\in(\e,1]\}$ with $\e\gg h$, such that $\tilde\Phi|_{\tilde S}$ is large in comparison to $h$. If the (Lebesgue) measure of $\nabla v(\tilde S)$ is large compared to
        $h$, then the integral \eqref{eq:5} is large too. Since $\nabla
        v(\tilde S)$ contains
        $\nabla v(\tilde S\cap \Gamma)=\tilde S\cap \Gamma$ by the boundary conditions, we conclude that $\nabla
        v(\tilde S)$ needs to be close to the one dimensional curve $\tilde S\cap \Gamma$, but of small measure, see Figure \ref{fig:gradv}. Now consider the
        preimages of angles under $\nabla v$, i.e., for $\varphi\in I$, the sets
        \[l_\varphi:=(\nabla
          v)^{-1}(\{r e_\varphi:r\in(0,\infty)\})\cap \tilde S\,,
        \]
        where $e_\varphi:=(\cos\varphi,\sin\varphi)$.

\begin{figure}[h]
  \begin{center}
    \includegraphics[height=5cm]{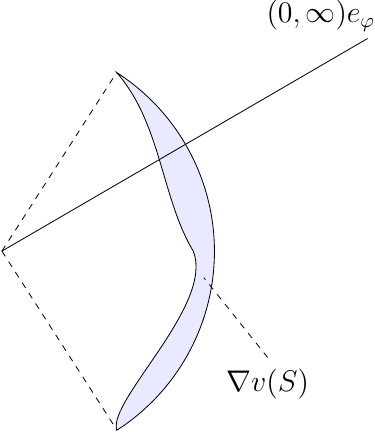}
  \end{center}
  \caption{The approximately one-dimensional gradient image $\nabla v(S)$.\label{fig:gradv}}
  \end{figure}
\begin{figure}[h]  \begin{center}
\includegraphics[height=5cm]{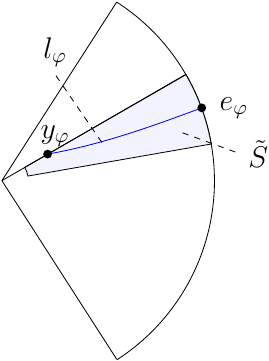}
\caption{The preimage of $(0,\infty)e_\varphi$ in $\tilde S$ under $\nabla v$ contains one connected component that starts at $e_\varphi$ and exits the domain  at $y_\varphi$. \label{fig:lphi}} 
\end{center}
\end{figure}

By the smallness of the measure of $\nabla v(\tilde S)$,  $|\nabla v|$ should be close to 1 on $l_\varphi$ for most
        $\varphi$.  We note
        that, again by the boundary conditions,  $l_\varphi$ that can be thought of as a curve starting in
        $e_\varphi$ and exiting $\tilde S$ either at one of the lateral boundaries (this case is sketched in Figure \ref{fig:lphi}) or at  $\partial\tilde S\cap\{|x|=\e\}$. 
\begin{figure}
  \begin{center}
    \includegraphics[height=5cm]{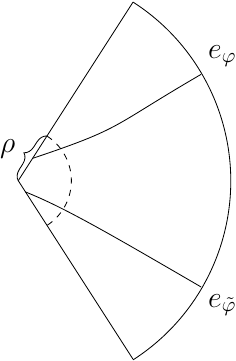}
\caption{If there are two preimages of angles $l_\varphi,l_{\tilde\varphi}$ such that a) they do not deviate much from the rays connecting the origin with $e_\varphi,e_{\tilde\varphi}$ respectively for radii larger than $\rho$ and b) $|\nabla v|$ is close to 1 on these curves, then the bending energy density $|\nabla^2 v|^2$ at radius $\rho$ is bounded from below, since $\nabla v\approx e_\varphi$ on $l_\varphi$ and $\nabla v\approx e_{\tilde\varphi}$ on $l_{\tilde\varphi}$.\label{fig:conc}}
  \end{center}
\end{figure}

In the first case, i.e., if the exit point from $\tilde S$ (which we call $y_\varphi$, see again Figure \ref{fig:lphi}) is at a distance larger than $\e$ from the origin, $|\nabla v|$ must have decreased along $l_\varphi$ by an amount of order $\e$. (This is the core of the argument and will be proved in Lemma \ref{lem:deficitestim} below; the convexity of $v$ is crucial for this step.) 
      This observation can be translated into a lower bound of the integral of $\det\nabla^2 v$ over  $l_\varphi$.
      There cannot be many such angles $\varphi$, otherwise the value of the integral \eqref{eq:5} becomes too large. Thus for most $\varphi$ we must have $|y_\varphi|=\e$. This implies that there is a large amount of bending energy density at every radius $\rho>\e$, see Figure \ref{fig:conc}. A judicious choice of $\e$ yields the desired lower bound.


\renewcommand{\e}{{h^*}}

\subsection*{Notation}
     We set
       \[
         \beta \coloneqq \left(\log \frac1h\right)^{-1}\,,\qquad \e \coloneqq h\left(\log\frac1h\right)^{6}\,,
         \]
         We work with the reduced `sector'
         \[
             S_h = \{x = re_\varphi \in S\,:\, \varphi \in
             (-\alpha+\beta,\alpha-\beta), r \geq \e\}
         \] that is 
         missing the tip and two thin outer sectors. For $M\subset\R^2$, we write  $2M=\{2x:x\in M\}$. Apart from the function
         $\arg: B(0,1)\setminus \{(x_1,0):x_1\leq 0\}\to (-\pi,\pi)$ introduced
         above, we will also use the function
         $\arg_{S^1}:B(0,1)\setminus\{0\}\to S^1$ defined by $x\mapsto
         x/|x|$. Subsets of $(-\pi,\pi)$ will be identified with their image in
         $S^1$ under the map $\varphi\mapsto e_\varphi:=(\cos\varphi,\sin\varphi)$.

  The
symbol ``$C$'' is used as follows: A statement such as ``$f\leq Cg$'' is shorthand for
``there exists a constant $C>0$ that only depends on $\alpha$ such that $f\leq
Cg$''. The value of $C$ may change within the same line. 
For $f\leq Cg$, we also write $f\lesssim g$.
The symbol $\H^d$ denotes the $d$-dimensional Hausdorff measure, while $\L^d$ is
the $d$-dimensional Lebesgue measure.

\section{Lower bound: Preliminaries.}

\subsection{Smooth approximation}

Here we show that we can assume $v\in \mathcal{A}$ to be smooth and strictly convex.

\begin{lemma}\label{lemma: smooth}
Let $h>0$, $\alpha \in (0,\pi/2)$.
\begin{itemize}
\item [(i)] Whenever $v,\tilde v\in \mathcal{A}$ then
\begin{equation}\label{eq: determinant continuity}
\|\det \nabla^2 Ev - \det \nabla^2 E\tilde v\|_{W^{-2,2}(2S)} \lesssim \|\nabla^2 v - \nabla^2 \tilde v\|_{L^2(S)} \|\nabla^2 v + \nabla^2 \tilde v\|_{L^2(S)}.
\end{equation}

\item [(ii)] Let $v\in \mathcal{A}$, $\delta>0$. Then there exists a strictly convex and smooth $v_\delta\in \mathcal{A}$ with
\begin{equation}\label{eq: smooth energy}
\|v-v_\delta\|_{W^{2,2}(S)} \leq \delta , \qquad J_h(v_\delta) \leq J_h(v) + \delta.
\end{equation}
\end{itemize}
\end{lemma}


\begin{proof}

To show (i), we use the identity \eqref{eq: curl curl}, which holds distributionally for $u\in W^{2,2}$. Fix a test function $\phi\in C_c^\infty(2S)$ and integrate by parts
\[
\begin{aligned}
\int_{2S} (\det \nabla^2 Ev - \det \nabla^2  E\tilde v)\phi \d x &= -\frac12 \int_{2S} \underbrace{(\nabla Ev \otimes \nabla E v - \nabla E\tilde v \otimes \nabla E\tilde v)}_{=0\text{ on } 2S\setminus S}:\cof \nabla^2 \phi \d x\\
&\leq  \|\nabla v - \nabla \tilde v\|_{L^4(S)} \|\nabla v + \nabla \tilde v\|_{L^4(S)} \|\nabla^2\phi\|_{L^2(S)}. 
\end{aligned}
\] 
Now we use the Sobolev-Poincar\'e inequality $\|\nabla w\|_{L^4(S)} \lesssim \|\nabla^2 w\|_{L^2(S)} + \|T \nabla w\|_{L^4(\Gamma)}$, where ``$T$'' denotes the trace operator. We observe $\|\nabla v-\nabla \tilde v\|_{L^4(\Gamma)} = 0$ and $\|\nabla v+ \nabla \tilde v\|_{L^4(\Gamma)} = O(1)\lesssim \|\nabla^2 v + \nabla^2 \tilde v\|_{L^2(S)}$, courtesy of the boundary values \eqref{eq:6}. Taking the supremum over all $\phi$ with $\int_{2S}|\nabla^2 \phi|^2 \d x \leq 1$ yields \eqref{eq: determinant continuity}.

Now to show (ii), we first extend $v$ to $V:2S \to \R$ by
\[
V(x) := \begin{cases}
v(x)  &x\in S\\
\frac{|x|^2+1}{2}&\text{otherwise.}
\end{cases}
\] 

We note that since $v\in \mathcal{A}$, $V$ is convex. We introduce a small parameter $\rho>0$ and define a function $v_\rho:S\to \R$ by
\[
v_\rho(x):= (V\ast \eta_\rho)(x+ a_\rho e_1).
\]

Here we convolve $V$ with $\eta_\rho(x):= \rho^{-d} \eta(x/\rho)$, where $\eta\in C_c^\infty(B(0,1))$ is a symmetric standard mollifier, and $a_\rho>0$ is the smallest number $a>0$ such that
\[\dist(S+ae_1, \partial S\setminus \Gamma) \geq \rho\quad\text{  and }\quad\dist(\Gamma+ae_1, S) \geq \rho\,.
\]
A valid choice for $a$ is given e.g.~by $a_\rho=\max(\rho/\sin\alpha,2 \sqrt{\rho})$. (The first number ensures the first condition, and the second number the second condition, even in the worst case $\alpha = \pi/2$.) The choice of $a_\rho$ guarantees that $v_\rho$ is well-defined, convex and smooth in $S$.

Since $\nabla v_\rho \neq x$ on $\Gamma$, we need to modify $v_\rho$ one more time to enforce the boundary values. We observe that
\[
(|\cdot|^2 \ast \eta_\rho)(x) = |x|^2 + \int_{B(0,\rho)} |y|^2 \eta_\rho(y)\d y = |x|^2 + b \rho^2,
\]
 where  $b>0$ depends only on  $\eta$. Now
 for $|x| \geq 1$ we have $B(x+a_\rho e_1,\rho) \cap S = \emptyset$ and hence
\[
\begin{aligned}
  v_\rho(x) &=  (V\ast \eta_\rho)(x+a_\rho e_1)\\
  &= \left(\frac{1+|\cdot|^2}{2} \ast \eta_\rho \right)(x+a_\rho e_1)\\
  &= \frac{1+ b \rho^2 + |\cdot|^2}{2}(x+a_\rho e_1)\\
 &=  \frac{1+ b\rho^2 + |x|^2 + 2a_\rho x_1 + a_\rho^2}{2},
\end{aligned}
\]
and
\[
\nabla v_\rho(x) = x+a_\rho e_1.
\]
 We thus set
\[
w_\rho(x):= v_{\rho}(x) - a_\rho x_1 - \frac{b \rho^2 + a_\rho^2}{2}, 
\]
and see immediately that $w_\rho\in \mathcal{A}$. 

It is straightforward to check that $w_\rho \to v$ in $W^{2,2}(S)$ as $\rho \to 0$. By \eqref{eq: determinant continuity}, we then have $J_h(w_\rho) \to J_h(v)$. 

Note that $w_\rho$ is in general not strictly convex, but we may replace it by $(1-\rho)w_\rho + \rho\frac{1+|\cdot|^2}{2}$, which is. It is trivial to check that all the above estimates still hold.

The result is obtained by picking $\rho = \rho(\delta)$ small enough.
\end{proof}

\begin{definition}
  Whenever $v\in \mathcal{A}$ such that $v$ is
  smooth and strictly convex on $S$, we write $v\in \mathcal A^*$.
\end{definition}

\subsection{Curves of constant slope}
In this subsection,  we assume that $v\in \mathcal A^*$ is fixed.

 For $\varphi\in (-\alpha + \beta, \alpha - \beta) $, let
 \[
   \begin{split}
     S_{h,\varphi}&=\{r(\cos\tilde \varphi,\sin\tilde\varphi):\e<r<1,\,\varphi-\beta<\tilde\varphi<\varphi+\beta\}\\
     \bar L_\varphi&=\nabla v^{-1}\left(\{re_\varphi:r\in(0,\infty)\}\right)\cap
     S_{h,\varphi}
   \end{split}
 \]

 \begin{lemma}
   \label{lem:Lstruc}
For every $\varphi\in (-\alpha+2\beta,\alpha-2\beta)$, $\bar L_\varphi$ is a countable  union of mutually disjoint smooth curves ending either on $\partial S_{h,\varphi}$ or in the critical point of $v$.
 \end{lemma}
 \begin{proof}
   This is an immediate consequence of the fact that the gradient of a strictly convex smooth map is a smooth diffeomorphism.
 \end{proof}

\begin{definition}
  \label{def:Lvarphi}
We denote by $L_\varphi$ the connected component of $\bar L_\varphi$ containing $e_\varphi$. By  Lemma \ref{lem:Lstruc}, $L_\varphi$ is a smooth curve. One of its endpoints is $e_\varphi$; the other endpoint will be denoted by $y_\varphi$, which lies either on $\partial S_{h,\varphi}$ or in the critical point of $v$.
\end{definition}

\begin{definition}
  An angle $\varphi\in (-\alpha+\beta,\alpha-\beta)$ is called good if
  \[
    |y_\varphi|=\e\qquad \text{ and }\qquad |\nabla v(y_\varphi)|\geq 1-\beta\,.
    \]
  Otherwise, $\varphi$ is called bad. Let $\mathcal B$ denote the set of bad angles.
  \end{definition}


\subsection{Definition of a suitable test function for the membrane term}
\label{sec:defin-suit-test}
Let $\eta_1\in C_c^\infty((h,2))$ with $\eta_1\geq 0$ and $\eta_1=1$ on
$(\e,\frac32)$, and $|\eta_1'|\lesssim (\e)^{-1}$,
$|\eta_1''|\lesssim(\e)^{-2}$. Let $\eta_2\in
C_c^\infty(-\alpha,\alpha)$ with $\eta_2\geq 0$,  $\eta_2=1$ on
$(-\alpha+ \beta,\alpha-\beta)$, and $|\eta_2'|\lesssim\beta^{-1}$, $|\eta_2''|\lesssim\beta^{-2}$. Set
\[
  \eta(x)=\eta_1(|x|)\eta_2\left(\arg(x)\right)\,,
\]
and
\[
   \Phi(x)=\eta(x) |x|\,.
\]

\begin{lemma}\label{lemma: test function}
  We have that $\Phi \in C_c^\infty(2S)$, $\Phi(x) = |x|$ in $S_h$, and
  \[
    \|\Phi\|_{W^{2,2}(2S)}\lesssim \left(\log\frac1h\right)^2\,.\]
\end{lemma}
\begin{proof}
The first two statements are obvious from our definition. To prove
  the last one,  we use polar coordinates $\rho,\theta$. We have $\Phi=\eta_1(\rho)\eta_2(\theta) \rho $ and calculate
  \[
    \begin{split}
    \nabla^2\Phi&=\rho\nabla^2\eta +\nabla\eta\otimes e_\theta + e_\theta\otimes \nabla\eta+\frac{\eta}{\rho} e_\theta^\bot\otimes e_\theta^\bot\\
    \nabla\eta&=\eta_1'\eta_2 e_\theta+\frac1\rho \eta_1\eta_2'\\
    \nabla^2\eta&= \eta_1''\eta_2 e_\theta\otimes e_\theta+\left(\eta_1'\eta_2'+\frac{\eta_1\eta_2'}{\rho}\right)\left(e_\theta\otimes e_\theta^\bot+e_\theta^\bot\otimes e_\theta\right)+\left(\frac{\eta_1'\eta_2}{\rho}+\frac{\eta_1\eta_2''}{\rho^2}\right)e_\theta^\bot\otimes e_\theta^\bot
  \end{split}
\]
and hence we get the estimate
\[
  \begin{split}
  \|\nabla^2\Phi\|_{L^2(2S)}^2&\lesssim
  \int_h^2\int_{-\alpha}^\alpha \left(|\rho\eta_1''\eta_2|^2+\left|\frac{\eta_1\eta_2'}{\rho}\right|^2+|\eta_1'\eta_2'|^2+|\eta_1'\eta_2|^2+\left|\frac{\eta_2''\eta_1}{\rho}\right|^2+\left|\frac{\eta_1\eta_2}{\rho}\right|^2\right)\rho\d\theta\d\rho\\
  &\lesssim 1+\left(\log\frac{1}{h}\right)^2+\log\frac{1}{h}+1+\left(\log\frac{1}{h}\right)^4+\log\frac1h\,.
\end{split}
  \]

The remaining estimates $\|\nabla \Phi\|_{L^2(2S)}^2 + \|\Phi\|_{L^2(2S)}^2 \lesssim \left(\log \frac1h \right)^4$ follow from the Poincar\'e inequality.
 \end{proof}



\section{Estimating the set of bad angles}

Here we show that if the energy is small, most angles $\varphi\in (-\alpha+\beta,\alpha-\beta)$ are good.

\begin{proposition}\label{prop: good angles}
  Assume that $v \in\mathcal A^*$  and that $J_h(v) \leq 3\alpha \log\frac1h$. Then the size of the set of bad angles can be estimated by
  \[
    \L^1(\mathcal B)\lesssim \left(\log \frac1h\right)^{-1} \,.
    \]
\end{proposition}

\begin{proof}
  We shall use the test function $\Phi$ constructed in Section \ref{sec:defin-suit-test} as follows: We extend $v$ to $Ev$. We use Lemma \ref{lemma: test function} to obtain
\begin{equation}\label{eq:8}
\begin{split}
h\sqrt{3\alpha \log \frac1h} \geq \|\det \nabla^2 Ev\|_{W^{-2,2}(2S)} \gtrsim & \frac{1}{\left(\log\frac1h\right)^{2}} \int_S \det \nabla^2 v(x) \Phi(x)\d x\\
\geq & \frac{1}{\left(\log\frac1h\right)^{2}} \int_{S_h} \det \nabla^2 v(x) |x| \d x
\end{split}
\end{equation}

Here we used the nonnegativity of both $\Phi$ and $\det \nabla^2 v$ to restrict the domain of integration.

Now note that bad angles in the lateral intervals $(-\alpha+\beta,-\alpha+2\beta)$ and $(\alpha - 2\beta, \alpha - \beta)$ contribute at most $2\beta$ to the size of $\mathcal B$. To deal with all other angles, we let
\[
  \begin{split}
    R_1&= \{\varphi\in (-\alpha+2\beta,\alpha-2\beta):|\nabla v(y_\varphi)|<1-\beta\}\\
    R_2&=\mathcal B \cap (-\alpha+2\beta, \alpha - 2\beta)\setminus R_1\,.
  \end{split}
\]
Note  that  $L_\varphi$ exits $S_{h,\varphi}$ laterally for $\varphi\in R_2$. 
Starting from  \eqref{eq:8} we obtain the following estimate,
      \[
      \begin{split}
      h\left(\log\frac1h\right)^{5/2} &\gtrsim  \int_{S_h} \det \nabla^2 v(x)|x|\,\d x\\
   &\gtrsim  \e \L^2(\nabla v[S_h]) \\
   &\gtrsim  \e\int_{R_1} 1-|\nabla v(y_\varphi)|^2\,\d\varphi\\
   &\gtrsim \e\beta\L^1(R_1)\,.
\end{split}
\]
Here  we have used the fact that $\nabla v$ is a diffeomorphism 
thanks to the convexity and smoothness of $v$, so that in particular $|\nabla v| \geq |\nabla v(y_\varphi)|$ on $L_\varphi$. Also, to obtain the third inequality, we have used polar coordinates in the codomain of $\nabla v$.
 
Inserting the definition of $\e, \beta$, we see that
\[
\L^1(R_1)\lesssim \left(\log\frac1h\right)^{-5/2}.
\]

 
Again starting from \eqref{eq:8}, using a change of variables, and the fact that $|\nabla v(y_\varphi)|\geq 1- \beta \geq \frac12$ for $\varphi\in R_2$, we have
\[
 \begin{aligned}
 h\left(\log\frac1h\right)^{5/2} \gtrsim &  \int_{S_h} \det \nabla^2v(x)|x|\,\d x\\
 = & \int_{\nabla v[S_h]} |(\nabla v)^{-1}(y)|\,\d y\\
 \gtrsim & \int_{R_2} \int_{|\nabla v(y_\varphi)|}^1 |(\nabla v)^{-1}(re_\varphi)| r \,\d r\,\d \varphi\\
 \geq & \frac12 \int_{R_2 } \int_{|\nabla v(y_\varphi)|}^1 (\nabla v)^{-1}(re_\varphi)\cdot e_\varphi  \,\d r\,d\varphi\\
 = & \frac12 \int_{R_2} \int_0^{|L_\varphi|} -\gamma_\varphi(s) \cdot e_\varphi \frac{\d}{\d s} |\nabla v(\gamma_\varphi(s))|\,\d s\,\d \varphi. 
 \end{aligned}
 \]
 
 Here $\gamma_\varphi:[0,|L_\varphi|] \to S_h$ is the arc-length parametrization
 of $L_\varphi$ starting in $\gamma_\varphi(0) = e_\varphi$ and ending in
 $\gamma_\varphi(|L_\varphi|) = y_\varphi$. We note that $\frac{\d}{\d s}|\nabla
 v(\gamma_\varphi(s))| < 0$ by the strict convexity of $v$. Since for all
 $\varphi\in R_2$, we have that $|\arg y_\varphi-\varphi|=\beta$, (i.e., $L_\varphi$ exits $S_{h,\varphi}$ laterally) we can use Lemma \ref{lem:deficitestim} below to obtain a uniform lower bound for the interior integral:
 
\[
\int_0^{|L_\varphi|} -\gamma_\varphi(s) \cdot e_\varphi \frac{\d}{\d s} |\nabla
v(\gamma_\varphi(s))|\,\d s  \gtrsim   \beta^2 \e\,.
\]

Hence we get
\[
  \L^1(R_2)\lesssim \frac{h\left(\log\frac1h\right)^{5/2}}{\e \beta^2}\lesssim \left(\log \frac1h\right)^{-3/2}\,,
\]
and summing up the lengths of the three components of $\mathcal{B}$,
\[
\L^1(\mathcal{B}) \leq 2\beta + \L^1(R_1) + \L^1(R_2) \lesssim \left(\log \frac1h\right)^{-1},
\]
which proves the proposition.
  \end{proof}

  \begin{lemma}
    \label{lem:deficitestim}
Let $\varphi,\tilde\varphi\in(-\alpha,\alpha)$ be such that $\arg y_\varphi=\tilde\varphi$ and $|y_\varphi|>\e$. Then

    \begin{equation}
    \label{eq:4}
    -\int_0^{|L_{\varphi}|}\gamma_\varphi(s)\cdot e_\varphi\frac{\d}{\d s}|\nabla v(\gamma_\varphi(s))|\d s\gtrsim \e (\varphi-\tilde\varphi)^2 \,.
  \end{equation}

\end{lemma}

\begin{proof}
  We will repeatedly use the fact that $\nabla v(\gamma_\varphi(s))=e_\varphi |\nabla v(\gamma_\varphi(s))|$.
By an integration by parts,
\[
  \begin{split}
  -\int_0^{|L_\varphi|}\gamma_\varphi(s)\cdot e_\varphi\frac{\d}{\d s}|\nabla v(\gamma_\varphi(s))|\d s &= \int_0^{|L_\varphi|} \gamma_\varphi'\cdot e_\varphi|\nabla v(\gamma_\varphi(s))| \d s-\big[\gamma_\varphi(s)\cdot e_\varphi |\nabla v(\gamma_\varphi(s))|\big]_0^{|L_\varphi|}\\
  &=(v(y_\varphi)-1)-(e_\varphi\cdot y_\varphi|\nabla v(y_\varphi)|-1)\\
  &=v(y_\varphi)-y_\varphi\cdot e_\varphi |\nabla v(y_\varphi)|\,.
\end{split}
  \]
  Using the convexity of $v$, we have that
  \[
    \begin{split}
    v(y_\varphi)&\geq v(e_{\tilde\varphi})+(y_\varphi- e_{\tilde\varphi})\cdot e_{\tilde \varphi}\\
    &=y_\varphi\cdot e_{\tilde\varphi}\,.
  \end{split}
    \]
    Inserting this in the previous equation, we obtain
 \[
   \begin{split}
-\int_0^{|L_\varphi|}\gamma_\varphi(s)\cdot e_\varphi\frac{\d}{\d s}|\nabla v(\gamma_\varphi(s))|\d s
    &\geq y_\varphi\cdot (e_{\tilde\varphi}-e_\varphi |\nabla v(y_\varphi)|)\\
    &\geq \e e_{\tilde\varphi}\cdot (e_{\tilde\varphi}-e_\varphi)\\
    &\gtrsim \e (\varphi-\tilde \varphi)^2\,,
  \end{split}
  \]
  proving our claim.
\end{proof}

\section{Proof of Theorem \ref{thm:mainma}}


\begin{proof}[Proof of  Theorem \ref{thm:mainma}]
  We first show the upper bound $\inf_{v\in \mathcal{A}} J_h(v) \leq 2\alpha \log\frac1h + C$.

Take the function $v_h\in \mathcal{A}$ from the proof of Corollary \ref{cor:FvK},
\[
v_h(x) := \begin{cases}
|x| &|x| \geq h\\
\frac{|x|^2}{2h} + \frac{h}2 &|x|<h. 
\end{cases}
\]

The bending energy was already estimated in \eqref{eq: FvK hessian}. We now turn to calculating 
\begin{equation}\label{eq:9}
  \begin{split}
    \frac1{h^2} \|\det \nabla^2 v_h\|_{W^{-2,2}(2S)}^2 &= \frac1{h^2}\sup_\phi \left( \int_{2S} \det\nabla^2 v_h\phi(x) \d x \right)^2\\
  &  = \frac1{h^2}\sup_\phi \left( \int_{S\cap B(0,h)} \frac1{h^2}\phi(x) \d x \right)^2,
\end{split}
\end{equation}

where the supremum is taken over all $\phi \in C_c^\infty(2S)$ with $\|\nabla^2\phi\|_{L^2(2S)} \leq 1$.
We use polar coordinates $\rho,\theta$ to integrate $\partial_\theta^2\phi$ in angular direction and obtain
\[
  \begin{split}
  \int_{B(0,h)\cap S} \phi(x)\d x&= \int_0^h\rho\d\rho \int_{-\alpha}^\alpha \d \varphi_1\int_{-\alpha}^{\varphi_1}\d\varphi_2\int_{-\alpha}^{\varphi_2}\d\theta \partial_{\theta}^2\phi(\rho,\theta)\\
&  \lesssim \alpha^2 \int_0^h\rho\d\rho \int_{-\alpha}^\alpha\d\theta |\rho^{-2}\partial_{\theta}^2\phi(\rho,\theta)|\rho^2\\
&\lesssim \alpha^2 \|\rho^2\|_{L^2(B(0,h)\cap S)}\|\nabla^2\phi\|_{L^2(S)}\\
&\lesssim \alpha^2 h^3\,.
\end{split}
\]
Inserting in \eqref{eq:9} yields
\[
\frac1{h^2} \sup_\phi\left( \int_{S\cap B(0,h)} \frac1{h^2}\phi(x) \d x \right)^2 \lesssim 
\alpha^4.
\]

In total, $J_h(v_h) \leq 2\alpha \left(\log\frac1h + 2\right) + C\alpha^4$.
\medskip

Now we prove the lower bound.

\medskip

  Assume that $v_0\in\mathcal A$ satisfies the upper bound $J_h(v)\leq 2\alpha\left(\log\frac1h+C\right)$. By Lemma
\ref{lemma: smooth}, we may consider instead $v\in\mathcal A^*$ with $J_h(v)\leq J_h(v_0)+\delta$ with $\delta$ arbitrarily small.
We claim that for  $s\in (\e,1)$,   the curve
\[
  \begin{split}
    \gamma_s:(-\alpha+\beta,\alpha-\beta)&\to \R^2\\
    \varphi&\mapsto \nabla v(se_\varphi)
  \end{split}
\]
satisfies

  \begin{equation}\label{eq:1}
  \H^1(\mathrm{Im}\gamma_s)\geq 2\alpha-C\left(\log \frac1h\right)^{-1}\,.
\end{equation}
Indeed, for every $\varphi\not\in \mathcal B$, we
 have that
  \[
    \mathrm{Im}\gamma_s \cap \arg^{-1}(\varphi)\setminus
    B(0,1-\beta)\neq\emptyset\,. \]
   Hence, using Proposition \ref{prop:
    good angles}, we have that
  \[
    \begin{split}
      \H^1(\mathrm{Im}\gamma_s)&\geq
      \H^1\left(\gamma_s\left((-\alpha+2\beta,\alpha-2\beta)\setminus\mathcal
          B\right)\right)\\
      &\geq (1-\beta) \H^1\left((-\alpha+2\beta,\alpha-2\beta)\setminus\mathcal
        B\right)\\
      &\geq 2\alpha-C\left(\log \frac1h\right)^{-1}\,.
    \end{split}
    \]




Now we estimate the
bending energy as follows, using Jensen's inequality and \eqref{eq:1}:
\[
  \begin{split}
    \int_{S_h^1} |\nabla^2 v|^2\,\d x &\geq \int_{-\alpha+\beta}^{\alpha-\beta}\int_\e^1
    \frac{1}{s^2}|\partial_\varphi\nabla
    v(se_\varphi)|^2 s \d s \d\varphi\\
    &\geq \int_\e^1 \frac{\d s}{s}
    \frac{1}{2\alpha}\left(\int_{-\alpha+\beta}^{\alpha-\beta}|\partial_\varphi
      \nabla v(s e_\varphi)|\d\varphi\right)^2\\
    &\geq \int_\e^1 \frac{\d s}{2\alpha s}
    \left(\max \left(2\alpha-C\left(\log\frac1h\right)^{-1},0\right)\right)^2\\
    &\geq 2\alpha\log\frac1h-C\log\log\frac1h\,.
  \end{split}
\]
This completes the proof of  Theorem \ref{thm:mainma}.

\end{proof}

\renewcommand{\e}{\varepsilon}

\section*{Acknowledgments}
This work has been supported by  Deutsche Forschungsgemeinschaft (DFG, German
Research Foundation) as part of project 350398276.

\bibliographystyle{alpha}

\bibliography{ccsector}

\end{document}